\documentclass[11pt,reqno]{amsart}

\usepackage{amsmath,amssymb,amsthm,mathtools}
\usepackage[margin=1.15in]{geometry}
\usepackage{xcolor}
\definecolor{linkblue}{rgb}{0.05,0.15,0.45}
\usepackage[colorlinks=true,linkcolor=linkblue,citecolor=linkblue,urlcolor=linkblue]{hyperref}
\usepackage{enumitem}

\theoremstyle{plain}
\newtheorem{theorem}{Theorem}[section]
\newtheorem{lemma}[theorem]{Lemma}
\newtheorem{proposition}[theorem]{Proposition}
\newtheorem{corollary}[theorem]{Corollary}
\newtheorem*{thmA}{Theorem~\ref{thm:final}}
\theoremstyle{remark}

\newcommand{\defn}{\textbf}

\newcommand{\Ft}{\mathbb{F}_2}
\newcommand{\Fp}{\mathbb{F}_p}
\newcommand{\ad}{\operatorname{ad}}
\newcommand{\gr}{\operatorname{gr}}
\newcommand{\spann}[1]{\langle #1\rangle}

\title[Restricted Lie algebras with isomorphic enveloping algebras]
{Non-isomorphic restricted Lie algebras\texorpdfstring{\\}{ }with isomorphic
restricted enveloping algebras}

\author[X. García-Martínez]{Xabier Garc\'ia-Mart\'inez}
\address{CITMAga \& Universidade de Santiago de Compostela, Departamento de
Matem\'aticas, E--15782 Santiago de Compostela, Spain}
\email{xabier.garcia@usc.gal}
\thanks{This work was supported by the project PID2024-155502NB-I00 granted by MICIU/AEI/10.13039/501100011033 and by Xunta de Galicia through the Competitive Reference Groups (GRC), ED431C 2023/31.}


\begin{document}

\begin{abstract}
Let $p$ be a prime. For every field $F$ of characteristic $p$ we exhibit pairs
of non-isomorphic finite-dimensional $p$-nilpotent restricted Lie algebras $L$
and $H$ over $F$ whose restricted enveloping algebras $u(L)$ and $u(H)$ are
isomorphic as $F$-algebras. Such pairs exist in every dimension at least
$p+5$, with $\dim L'=p$ and $\dim H'=p+1$. Thus, the restricted isomorphism problem has a negative answer over every
field of positive characteristic, even for $p$-nilpotent algebras over
perfect fields. 
\end{abstract}

\maketitle

\section{Introduction}\label{sec:intro}

Whether a finite group is determined by its group algebra is one of the
classical questions of representation theory. In general, the answer
is negative: the dihedral and the quaternion group of order eight have
isomorphic complex group algebras, and even over the integers the answer is
negative, by Hertweck's celebrated counterexample~\cite{Hertweck}. A particularly delicate form of the question is the modular one: 
if $G$ and $H$ are finite
$p$-groups and $F$ is a field of characteristic $p$, does $FG\cong FH$ imply~$G\cong H$? This is the \emph{modular isomorphism problem}, which appears in Brauer's 1963 problem list~\cite{Brauer}, and for half a century it
accumulated some positive results: Deskins proved it for abelian $p$-groups
\cite{Deskins}, Passi and Sehgal for groups of class two and exponent $p$
\cite{PassiSehgal}, and many further families followed (we refer to the
surveys of Sandling~\cite{Sandling} and Margolis~\cite{Margolis}). However, in the
article~\cite{GLMdR}, Garc\'ia-Lucas, Margolis and del~R\'io found a negative solution:
there exist non-isomorphic $2$-groups of order
$2^9$ with isomorphic group algebras over every field of characteristic two.

Modular group algebras are tied to restricted Lie algebras in a precise way.
Jennings~\cite{Jennings} showed that the dimension subgroups of a finite
$p$-group $G$ assemble into a graded restricted Lie algebra, and Quillen
\cite{Quillen} identified the associated graded algebra of $FG$ with respect
to the augmentation filtration as its restricted enveloping algebra. 
This leads to the following restricted analogue of the isomorphism problem: 
\begin{description}
	\item[Restricted Isomorphism Problem]
		Let $L$ and $H$ be two restricted Lie algebras over a field $F$ of positive characteristic.
		If $u(L)\cong u(H)$ as $F$-algebras, must $L\cong H$ as restricted Lie algebras?
\end{description}
Usefi's survey~\cite{UsefiSurvey} raises it twice among its open problems:
Problem~6 asks for a counterexample, and Problem~9 asks whether the answer is
positive at least for $p$-nilpotent restricted Lie algebras over perfect
fields.

The known results are closely related to the group case. In the article~\cite{UsefiPAMS},
Usefi proved that~$u(\mathfrak{g})$ determines $\mathfrak{g}$ when it is $p$-nilpotent and metacyclic
over a perfect field. In~\cite{UsefiPacific}, he proved that $u(\mathfrak{g})$ determines
$\mathfrak{g}$ when it is $p$-nilpotent and abelian over a perfect field, as well as when
$\mathfrak{g}$ is abelian and finite-dimensional over an algebraically closed field. In the
same article he also established that $u(\mathfrak{g})$ determines the
dimension-subalgebra quotients $D_i(\mathfrak{g})/D_{i+1}(\mathfrak{g})$, hence the associated
graded restricted Lie algebra $\gr \mathfrak{g}$, and, for $p$-nilpotent algebras,
the nilpotency class up to an error of one, and that
the quotient
$\mathfrak{g}/({\mathfrak{g}'}^{p}+\gamma_3(\mathfrak{g}))$ is determined as well. On the other hand, for the ordinary universal
enveloping algebra $U(\mathfrak{g})$ the answer in prime characteristic
has long been known to be negative. Examples of non-isomorphic Lie algebras
with isomorphic universal enveloping algebras have been known at least since
Riley and Usefi~\cite[Example~A]{RileyUsefi}. Schneider and Usefi
\cite{SchneiderUsefi} proved that nilpotent Lie algebras of dimension at most
six are determined by $U(\mathfrak{g})$ whenever the characteristic of the field is not $2$ or $3$, exhibiting at the same time nilpotent
counterexamples of dimensions five and six in characteristics two and three,
and for every prime $p$ there is a pair of non-isomorphic metabelian
nilpotent Lie algebras of dimension $p+3$ whose universal enveloping algebras
are isomorphic even as Hopf algebras (see~\cite{SchneiderUsefi} and
\cite[Example~6.1]{UsefiSurvey}).

Furthermore, in characteristic zero, the answer for
nilpotent Lie algebras is positive: Problem~5 of~\cite{UsefiSurvey} was
recently settled, affirmatively, by Campos, Petersen, Robert-Nicoud and
Wierstra~\cite{CPRW}, who showed that a nilpotent Lie algebra over a field of
characteristic zero is determined by its universal enveloping algebra. The preceding results, 
however, concern the infinite-dimensional algebra $U(\mathfrak{g})$ and
say nothing about the restricted enveloping algebra~$u(\mathfrak{g})$: the
restricted isomorphism problem remained open in every characteristic.

The aim of this article is to settle the problem by constructing such
examples over every field of positive characteristic.

\begin{thmA}\label{thm:A}
Let $F$ be a field of characteristic $p>0$. For every integer $n\ge p+5$
there exist $p$-nilpotent restricted Lie algebras $L$ and $H$ over $F$ with
$\dim L=\dim H=n$,
\[
u(L)\cong u(H)\ \text{ as $F$-algebras},\qquad
\dim L'=p\neq p+1=\dim H' .
\]
\end{thmA}

As a consequence, the restricted isomorphism problem has a negative
answer over every field of positive characteristic, already within the class
of $p$-nilpotent restricted Lie algebras over perfect fields.
Neither the dimension of the derived subalgebra nor the lower central series
of a restricted Lie algebra is determined by its restricted enveloping
algebra. In particular, the answers to Problems~6 and~9 of
\cite{UsefiSurvey} are negative.

The manuscript is organised as follows. After some short preliminaries, in Section~\ref{sec:char2} we construct a parametric family of examples
over $\Ft$. Since this construction does not carry over verbatim to odd characteristic, in Section~\ref{sec:oddp} we present a new family of examples for any~$\Fp$. 
In Section~\ref{sec:fields} we put everything together and prove the main
theorem over an arbitrary field of positive characteristic.
We conclude in Section~\ref{sec:concl} with remarks on how the examples sit
among the determined invariants, on the non-restricted enveloping algebra, and on the computer search that produced the first
counterexample.

\section{Preliminaries}\label{sec:prelim}

Throughout the article $p$ is a prime and $F$ a field of characteristic $p$. 
All the results in this section can be found in~\cite{Jacobson} and in~\cite{SF}.
A \defn{restricted Lie algebra} over $F$ is a Lie algebra $\mathfrak{g}$ equipped with a
map $x\mapsto x^{[p]}$, usually called \defn{$p$-map}, such that
$\ad(x^{[p]})=(\ad x)^p$, $(\lambda x)^{[p]}=\lambda^p x^{[p]}$ and
\begin{equation}\label{eq:jacobsonformula}
(x+y)^{[p]}=x^{[p]}+y^{[p]}+\sum_{i=1}^{p-1}s_i(x,y),
\end{equation}
where the $s_i$ are the universal Lie polynomials of Jacobson's formula:
 $i\,s_i(x,y)$ is the coefficient of $t^{i-1}$ in $\ad(tx+y)^{p-1}(x)$. All we
shall use about them is that each $s_i(x,y)$ is a sum of iterated brackets of
$x$ and $y$. For $p=2$
formula \eqref{eq:jacobsonformula} is $(x+y)^{[2]}=x^{[2]}+y^{[2]}+[x,y]$.
Jacobson's formula also holds for the $p$-th power map of any associative
$F$-algebra, with $[r,r']\coloneqq rr'-r'r$. In particular, in any associative
$F$-algebra,
\begin{equation}\label{eq:square}
(r+r')^p=r^p+r'^p \text{ whenever } [r,r']=0,
\text{ and, for } p=2,
\Bigl(\sum_i r_i\Bigr)^{2}=\sum_i r_i^{2}+\sum_{i<j}\,[r_i,r_j].
\end{equation}
We shall use throughout, without further mention, that the commutator is a
derivation in each variable: $[r,r'r'']=[r,r']\,r''+r'\,[r,r'']$ and
$[rr',r'']=r\,[r',r'']+[r,r'']\,r'$.

We say that $\mathfrak{g}$ is \defn{$p$-nilpotent} if for every $x\in \mathfrak{g}$ there is an $N$
with $x^{[p]^N}=0$, where~$x^{[p]^N}$ denotes the $N$-fold iterate of the
$p$-map. A basic tool for constructing restricted structures is
\defn{Jacobson's theorem}: if $\mathfrak{g}$ is a Lie algebra over $F$ with basis $\{e_i\}$
and $h_i\in \mathfrak{g}$ are vectors satisfying $\ad h_i=(\ad e_i)^p$ for every $i$,
then $\mathfrak{g}$ carries a unique $p$-map with $e_i^{[p]}=h_i$.

The \defn{restricted enveloping algebra} of $\mathfrak{g}$ is
$u(\mathfrak{g})=U(\mathfrak{g})/\bigl(x^p-x^{[p]}:x\in \mathfrak{g}\bigr)$. By the
Poincar\'e--Birkhoff--Witt theorem for restricted enveloping algebras, if $\mathfrak{g}$ has basis
$e_1,\dots,e_n$ then the \emph{ordered restricted monomials}
$e_1^{\alpha_1}\cdots e_n^{\alpha_n}$ with $0\le\alpha_i\le p-1$ form a basis
of $u(\mathfrak{g})$. In particular $\dim u(\mathfrak{g})=p^{\,n}$ and $\mathfrak{g}$ embeds into $u(\mathfrak{g})$. The
algebra $u(\mathfrak{g})$ has the expected universal property: every homomorphism of
restricted Lie algebras from $\mathfrak{g}$ to an associative $F$-algebra $A$, viewed as
a restricted Lie algebra under the commutator and the $p$-map,
extends uniquely to an algebra homomorphism~$u(\mathfrak{g})\to A$.

\begin{lemma}[\cite{Jacobson}]\label{lem:pres}
Let $\mathfrak{g}$ be a restricted Lie algebra with basis $e_1,\dots,e_n$. Then, as an
associative algebra,
\begin{equation*}
u(\mathfrak{g})\;=\;F\langle e_1,\dots,e_n\rangle\Big/
\bigl(\,e_ie_j-e_je_i-[e_i,e_j],\quad e_i^{\,p}-e_i^{[p]}\,\bigr)_{i,j}\,.
\tag{P}\label{eq:pres}
\end{equation*}
\end{lemma}

\section{Characteristic two}\label{sec:char2}

Throughout this section we work over the field $\Ft$ of two elements. We fix an integer $k\ge3$ and set
\[
q\coloneqq 2^{k-1}\ \ (\text{so } q\ge4),\qquad n\coloneqq k+4 .
\]

Let $L_k$ be the $n$-dimensional restricted Lie algebra over $\Ft$ with basis
\[
a,\ b,\ b_2,\ a_2,\ v,\ b_4,\ b_8,\ \dots,\ b_{q},\ d
\qquad(\text{the } b_{2^j} \text{ for } 2\le j\le k-1,\ \text{and } d),
\]
$2$-map
\begin{gather*}
a^{[2]}=a_2,\quad a_2^{[2]}=d,\quad
b^{[2]}=b_2,\quad b_2^{[2]}=b_4,\\
b_{2^j}^{[2]}=b_{2^{j+1}}\ (2\le j\le k-2),
\quad b_{q}^{[2]}=d,\quad v^{[2]}=d^{[2]}=0,
\end{gather*}
with nonzero brackets
\[
[a,b]=v,\qquad [a_2,b]=b_4,\qquad [a,v]=b_4 .
\]
Let $H_k$ be the $n$-dimensional restricted Lie algebra over $\Ft$ with basis
\[
x,\ y,\ y_2,\ x_2,\ w,\ y_4,\ y_8,\ \dots,\ y_{q},\ c,
\]
$2$-map
\begin{gather*}
x^{[2]}=x_2,\quad x_2^{[2]}=c,\quad
y^{[2]}=y_2,\quad y_2^{[2]}=y_4,\\
y_{2^j}^{[2]}=y_{2^{j+1}}\ (2\le j\le k-2),
\quad y_{q}^{[2]}=c,\quad w^{[2]}=c^{[2]}=0,
\end{gather*}
and nonzero brackets
\[
[x,y]=w,\qquad [x,y_2]=y_4,\qquad [y,w]=y_4,\qquad
[x_2,y]=y_4+c,\qquad [x,w]=y_4+c.
\]
Each algebra is generated by its first two basis vectors, the two $2$-power towers
merge at the top, and the eighth power of the first generator vanishes.

\begin{lemma}\label{lem:valid2}
$L_k$ and $H_k$ are restricted Lie algebras of dimension $n=k+4$, generated
by $a,b$ and by $x,y$ respectively. Both have nilpotency class $3$, with
lower central dimension sequences $(n,2,1,0)$ for $L_k$ and $(n,3,2,0)$ for
$H_k$. In particular
\[
L_k'=\spann{v,b_4},\qquad H_k'=\spann{w,y_4,c}.
\]
\end{lemma}

\begin{proof}
By Jacobson's theorem it suffices to check the Jacobi identity and
$(\ad e)^2=\ad(e^{[2]})$ for each basis vector $e$.

Let us begin with $L_k$. The nonzero adjoint values are
\[
\ad a\colon \ b\mapsto v,\ v\mapsto b_4,\qquad
\ad b\colon \ a\mapsto v,\ a_2\mapsto b_4,\qquad
\ad a_2\colon \ b\mapsto b_4,\qquad
\ad v\colon \ a\mapsto b_4 .
\]
Then $(\ad a)^2$ maps $b\mapsto v\mapsto b_4$ and kills every other basis
vector, which is $\ad a_2=\ad(a^{[2]})$. Each of $\ad b,\ad a_2,\ad v$ maps
the basis into $\spann{v,b_4}$, which it kills, so its square is zero,
matching $\ad(b^{[2]})=\ad b_2=0$, $\ad(a_2^{[2]})=\ad d=0$,
$\ad(v^{[2]})=0$. All remaining basis vectors are central with central
squares. For the Jacobi identity: every bracket lies in $\spann{v,b_4}$, the
vector $b_4$ is central, and $[v,e]$ is nonzero only for $e=a$. Since $v$
occurs in $[e_i,e_j]$ only for $\{e_i,e_j\}=\{a,b\}$, a nonzero summand
$[[e_i,e_j],e_l]$ of a Jacobiator with distinct entries would need
$\{e_i,e_j\}=\{a,b\}$ and $e_l=a$, which is impossible.

Let us consider now $H_k$. The nonzero adjoint values are
\[
\ad x\colon \ y\mapsto w,\ y_2\mapsto y_4,\ w\mapsto y_4+c,\qquad
\ad y\colon \ x\mapsto w,\ x_2\mapsto y_4+c,\ w\mapsto y_4,
\]
\[
\ad y_2\colon \ x\mapsto y_4,\qquad
\ad x_2\colon \ y\mapsto y_4+c,\qquad
\ad w\colon \ x\mapsto y_4+c,\ y\mapsto y_4 .
\]
Then $(\ad x)^2\colon y\mapsto y_4+c$ (all other basis vectors $\mapsto0$)
equals $\ad x_2$, and $(\ad y)^2\colon x\mapsto y_4$ equals $\ad y_2$. Each of
$\ad y_2,\ad x_2,\ad w$ maps the basis into the central space
$\spann{y_4,c}$, so its square is zero, matching $\ad y_4=\ad c=\ad 0=0$. For
the Jacobi identity: every bracket lies in $\spann{w,y_4,c}$ with
$\spann{y_4,c}$ central, $[w,e]\neq0$ only for $e\in\{x,y\}$, and $w$ occurs
only in $[x,y]$. A nonzero Jacobiator summand with distinct entries would need
inner pair $\{x,y\}$ and outer entry in $\{x,y\}$, which is impossible.

Generation is clear in both cases by the $2$-map and the brackets. The lower
central series are
\begin{align*}
	L_k&\supset\spann{v,b_4}\supset\spann{b_4}\supset0,  & H_k&\supset\spann{w,y_4,c}\supset\spann{y_4,c}\supset0. 
\end{align*}
\end{proof}

\begin{corollary}\label{cor:noniso2}
The restricted Lie algebras $L_k$ and $H_k$ are not isomorphic. In fact they
are not isomorphic even as Lie algebras, since
$\dim L_k'=2\neq3=\dim H_k'$.
\end{corollary}

By Lemma~\ref{lem:pres}, the relations \eqref{eq:pres} hold in $u(H_k)$.
Inside $u(H_k)$ we may therefore write
\[
x^2=x_2,\qquad y^{2^j}=y_{2^j},\qquad x^4=x_2^{\,2}=c,
\qquad w=xy+yx,
\]
and we use this exponent notation from now on. All computations below take
place in $u(H_k)$. We define:
\[
\delta\coloneqq x^4+y^4,\qquad
g\coloneqq y^{2q-4}+y^{4q-8}.
\]

\begin{lemma}\label{lem:tool2}
In $u(H_k)$ the following hold:
\begin{enumerate}
\item $[x,y^2]=y^4$, $\ [x^2,y]=\delta$, $\ w^2=0$, $\ x^4=y^{2^k}$,
      $\ x^8=0$, and in particular $y^{4q}=0$;
\item $y^4$ and $x^4$, $g$ and $\delta$ are central. Moreover $\delta^2=y^8$;
\item $[x^2,y^2]=[x^2,w]=[y^2,w]=0$, $\ [x,w]=\delta$, $\ [y,w]=y^4$;
\item $g\,\delta=x^4$ and $x^4g^4=0$.
\end{enumerate}
\end{lemma}

\begin{proof}
(1) These are immediate from the defining relations~\eqref{eq:pres} of $u(H_k)$.

(2) $[y,y^4]=0$ and
$[x,y^4]=[x,y^2]\,y^2+y^2\,[x,y^2]=y^4y^2+y^2y^4=0$, so $y^4$ is central, since $x$ and $y$ generate $u(H_k)$. Then $[x,x^4]=0$ and
$[y,x^4]=[y,x^2]\,x^2+x^2\,[y,x^2]=\delta x^2+x^2\delta=0$, since $x^4$
commutes with $x^2$ and $y^4$ is central, $x^4$ is central. Products and
sums of central elements are central, and $g$ is a polynomial in $y^4$
because its exponents are divisible by $4$. Finally
$\delta^2=x^8+y^8+[x^4,y^4]=y^8$.

(3) $[x^2,y^2]=[x^2,y]\,y+y\,[x^2,y]=\delta y+y\delta=0$. Next
$[x^2,xy]=x\,[x^2,y]=x\delta$ and $[x^2,yx]=[x^2,y]\,x=\delta x$, so
$[x^2,w]=x\delta+\delta x=0$. Similarly $[y^2,xy]=[y^2,x]\,y=y^4y$ and
$[y^2,yx]=y\,[y^2,x]=y\,y^4$, so $[y^2,w]=0$. Finally
$[x,w]=x^2y+yx^2=[x^2,y]=\delta$ and $[y,w]=y^2x+xy^2=[x,y^2]=y^4$.

(4) Using $x^4=y^{2q}$ and $y^{4q}=0$,
\[
g\,\delta=(y^{2q-4}+y^{4q-8})(y^{2q}+y^{4})
= y^{4q-4}+y^{2q}+y^{6q-8}+y^{4q-4}
= y^{2q}=x^4,
\]
since the two terms $y^{4q-4}$ cancel and $6q-8\ge4q$ (as $q\ge4$). For the
second identity, $g^4=y^{8q-16}+y^{16q-32}$ by \eqref{eq:square}, and
$x^4g^4=y^{10q-16}+y^{18q-32}=0$ because $10q-16\ge4q$.
\end{proof}

\begin{theorem}\label{thm:char2}
For every $k\ge3$ the assignment
\[
\varphi(a)=X\coloneqq x,\qquad
\varphi(b)=Y\coloneqq x+y+x\,g
= x+y+x\bigl(y^{2^{k}-4}+y^{2^{k+1}-8}\bigr),
\]
extends to an isomorphism of $\Ft$-algebras
$\varphi\colon u(L_k)\to u(H_k)$.
\end{theorem}

\begin{proof}
	We define the images of the basis vectors of $L_k$ as the following words in $X$ and $Y$:
\begin{gather*}
f_a\coloneqq X,\qquad f_b\coloneqq Y,\qquad f_{b_2}\coloneqq Y^2,\qquad f_{a_2}\coloneqq X^2,\qquad
f_v\coloneqq [X,Y],\\
f_{b_{2^j}}\coloneqq Y^{2^j}\ (2\le j\le k-1),\qquad f_d\coloneqq X^4 .
\end{gather*}
The assignment extends to an
algebra homomorphism $\varphi\colon u(L_k)\to u(H_k)$ with $\varphi(e)=f_e$
for every basis vector $e$ of $L_k$ as soon as the elements $f_e$ satisfy all
the relations \eqref{eq:pres} of $u(L_k)$. Let us verify them:

Let us begin by checking the power relations. By construction $f_a^{\,2}=f_{a_2}$,
$f_b^{\,2}=f_{b_2}$, $f_{b_2}^{\,2}=f_{b_4}$,
$f_{b_{2^j}}^{\,2}=f_{b_{2^{j+1}}}$ for $2\le j\le k-2$,
$f_{a_2}^{\,2}=f_d$ and $f_d^{\,2}=X^8=x^8=0$, matching the $2$-map of~$L_k$. There are two remaining power relations: $f_v^{\,2}=0$ and
$f_{b_q}^{\,2}=f_d$.

Since $g$ is central, $[x,xg]=g\,[x,x]=0$, so
\[
f_v=[X,Y]=[x,y]+[x,xg]=w,
\qquad\text{whence}\qquad
f_v^{\,2}=w^2=0.
\]

By \eqref{eq:square} applied to $Y=x+y+xg$: the
squares of the summands are $x^2$, $y^2$ and $(xg)^2=x^2g^2$, and the
cross-brackets are $[x,y]=w$, $[x,xg]=0$ and $[y,xg]=g\,[y,x]=gw$. Hence
\begin{equation}\label{eq:Ysq}
Y^2=x^2+y^2+x^2g^2+(1+g)\,w .
\end{equation}

The four summands of \eqref{eq:Ysq} commute
pairwise: $g$ is central and $[x^2,y^2]=[x^2,w]=[y^2,w]=0$. So by \eqref{eq:square} their squares add up:
\[
Y^4=x^4+y^4+x^4g^4+(1+g)^2w^2=\delta+x^4g^4=\delta.
\]
Thus $f_{b_4}=\delta$ and
$Y^8=\delta^2=y^8$, so that
\[
f_{b_{2^j}}=Y^{2^j}=\bigl(Y^8\bigr)^{2^{j-3}}=y^{2^j}\quad(3\le j\le k-1):
\]
all the elements $f_{b_4},\dots,f_{b_q},f_d=x^4$ are central in $u(H_k)$. The
missing power relation follows:
\[
f_{b_q}^{\,2}=Y^{2^k}=\bigl(Y^{8}\bigr)^{2^{k-3}}
=\bigl(y^{8}\bigr)^{2^{k-3}}=y^{2^k}=x^4=f_d.
\]

Let us proceed now to check the brackets. In $L_k$ every basis vector other than
$a,b,b_2,a_2,v$ is central, and the corresponding images
$f_{b_4},\dots,f_{b_q},f_d$ are central in $u(H_k)$, hence all relations~\eqref{eq:pres} involving them hold, both sides being zero. Moreover
$[f_a,f_b]=f_v$ holds by construction, and
$[f_a,f_{a_2}]=[x,x^2]=0$, $[f_b,f_{b_2}]=[Y,Y^2]=0$ trivially. Let us compute the remaining pairs:
\begin{align*}
[f_a,f_{b_2}]&=[x,Y^2]=[x,y^2]+g^2\,[x,x^2]+(1+g)\,[x,w]
=y^4+(1+g)\,\delta\\
&=(y^4+\delta)+g\delta=x^4+x^4=0=f_{[a,b_2]},\\[2pt]
[f_{a_2},f_b]&=[x^2,Y]=[x^2,x]+[x^2,y]+g\,[x^2,x]=\delta
=f_{b_4}=f_{[a_2,b]},\\[2pt]
[f_a,f_v]&=[x,w]=\delta=f_{b_4}=f_{[a,v]},\\[2pt]
[f_b,f_v]&=[Y,w]=[x,w]+[y,w]+g\,[x,w]=\delta+y^4+g\delta
=\delta+y^4+x^4=\delta+\delta=0=f_{[b,v]},\\[2pt]
[f_{b_2},f_{a_2}]&=[Y^2,x^2]=0,\qquad
[f_{b_2},f_v]=[Y^2,w]=0,\qquad
[f_{a_2},f_v]=[x^2,w]=0.
\end{align*} 
In each case the value agrees with
the image of the corresponding bracket of $L_k$. All relations~\eqref{eq:pres} of $u(L_k)$
hold, so $\varphi$ is a well-defined algebra homomorphism.

To see the bijectivity, by Lemma~\ref{lem:valid2} and the
Poincar\'e--Birkhoff--Witt theorem the two algebras have dimension $2^{\,n}$,
so it suffices to show that $\varphi$ is surjective. As $L_k$ is generated by~$a,b$, 
so is $u(L_k)$ as an associative algebra, hence the image of $\varphi$
is the subalgebra $B$ generated by~$X=x$ and $Y$, and it is enough to show
$x,y\in B$. Certainly $x=X\in B$, and, the characteristic being two,
\[
\eta\coloneqq X+Y=y+xg\in B .
\]
Moreover, $\eta^4=y^4$. Indeed $X^2=x^2$ and $[X,Y]=w$, so \eqref{eq:square} gives $\eta^2=X^2+Y^2+[X,Y]=x^2+Y^2+w$. Squaring
once more, the cross terms vanish, since $w^2=0$,
$[x^2,w]=0$, and $[x^2,Y^2]=[Y^2,w]=0$ were shown
among the bracket relations above, whence
\[
\eta^4=X^4+Y^4=x^4+\delta=y^4 ,
\]
using $Y^4=\delta=x^4+y^4$. Now $g=(y^4)^{m}+(y^4)^{2m}$ with
$m\coloneqq 2^{k-2}-1$ (because $2q-4=4m$ and $4q-8=8m$), so
$g=(\eta^4)^{m}+(\eta^4)^{2m}\in B$. Therefore $xg\in B$ and
$y=\eta+xg\in B$. Thus $B=u(H_k)$, and $\varphi$ is an isomorphism.
\end{proof}

\section{Odd characteristic}\label{sec:oddp}

Throughout this section $p$ is an odd prime and we work over the prime field
$\Fp$. The characteristic-two construction does not carry over verbatim:
the coefficient $2^{-1}$ below and the bound $(\ad y)^3=0$ used in
Lemma~\ref{lem:validp} are both special to odd $p$, just as the identities of
Lemma~\ref{lem:tool2} are special to $p=2$, but the mechanism is similar. We fix an integer $k\ge3$
and set
\[
n\coloneqq p+k+2,\qquad q\coloneqq p^{\,k-2}.
\]

Let $L_{p,k}$ be the $n$-dimensional restricted Lie algebra over $\Fp$ with
basis
\[
a,\ b,\ v_1,\ v_2,\ \dots,\ v_{p-1},\ a_p,\ b_p,\ b_{p^2},\ \dots,\
b_{p^{k-1}},\ d,
\]
$p$-map
\begin{gather*}
a^{[p]}=a_p,\quad a_p^{[p]}=d,\quad
b^{[p]}=b_p, \\ b_{p^j}^{[p]}=b_{p^{j+1}}\ (1\le j\le k-2), \quad
b_{p^{k-1}}^{[p]}=d, v_i^{[p]}=d^{[p]}=0,
\end{gather*}
and nonzero brackets
\[
[a,b]=v_1,\quad [a,v_i]=v_{i+1}\ (1\le i\le p-2),\quad
[a,v_{p-1}]=b_{p^2},\quad [a_p,b]=b_{p^2}.
\]
Let $H_{p,k}$ be the $n$-dimensional restricted Lie algebra over $\Fp$ with
basis
\[
x,\ y,\ u_1,\ u_2,\ \dots,\ u_{p-1},\ x_p,\ y_p,\ y_{p^2},\ \dots,\
y_{p^{k-1}},\ c,
\]
$p$-map
\begin{gather*}
x^{[p]}=x_p,\quad x_p^{[p]}=c,\quad
y^{[p]}=y_p, \\ y_{p^j}^{[p]}=y_{p^{j+1}}\ (1\le j\le k-2),\quad
y_{p^{k-1}}^{[p]}=c, u_i^{[p]}=c^{[p]}=0,
\end{gather*}
and nonzero brackets
\[
[x,y]=u_1,\quad [x,u_i]=u_{i+1}\ (1\le i\le p-2),\quad
[x,u_{p-1}]=y_{p^2},\quad [x_p,y]=y_{p^2},\quad [y,u_1]=c .
\]
For the remainder of this section we abbreviate $L\coloneqq L_{p,k}$ and
$H\coloneqq H_{p,k}$.

\begin{lemma}\label{lem:validp}
$L$ and $H$ are restricted Lie algebras of dimension $n=p+k+2$, generated by
$a,b$ and by $x,y$ respectively. Furthermore $\dim L'=p$ and
$\dim H'=p+1$.
\end{lemma}

\begin{proof}
The nonzero adjoint values are
\[
\ad x\colon \ y\mapsto u_1,\ u_i\mapsto u_{i+1}\ (i\le p-2),\ u_{p-1}\mapsto y_{p^2},
\qquad
\ad y\colon\ x\mapsto -u_1,\ x_p\mapsto -y_{p^2},\ u_1\mapsto c,
\]
\begin{gather*}
\ad u_1\colon\ x\mapsto -u_2,\ y\mapsto -c,\qquad
\ad u_i\colon\ x\mapsto -u_{i+1}\ (2\le i\le p-2),\\
\ad u_{p-1}\colon\ x\mapsto -y_{p^2},\qquad
\ad x_p\colon\ y\mapsto y_{p^2}.
\end{gather*}
Then $(\ad x)^p$ sends
$y\mapsto u_1\mapsto\cdots\mapsto u_{p-1}\mapsto y_{p^2}$ ($p$ steps) and
kills all other basis vectors, which is exactly $\ad x_p=\ad(x^{[p]})$. Each
of $\ad y$, $\ad u_i$, $\ad x_p$ maps the basis into
$\spann{u_1,\dots,u_{p-1},y_{p^2},c}$, and its second iterate already lands
in the central space $\spann{y_{p^2},c}$
($(\ad y)^2(x)=-[y,u_1]=-c$), so its third iterate vanishes. As $p\ge3$, all
the required $p$-th iterates are zero, matching
$\ad(y^{[p]})=\ad(u_i^{[p]})=0$ and $\ad(x_p^{[p]})=\ad c=0$. The chain
elements $y_{p^j},c$ are central with central $p$-images.

To check the Jacobi identity, let us note that every bracket lies in
$V=\spann{u_1,\dots,u_{p-1},y_{p^2},c}$, where $\spann{y_{p^2},c}$ is
central. The only elements of $V$ with nonzero adjoint are the $u_i$, which
act only on $x$ (and $u_1$ also on $y$), with central values
$u_{i+1},y_{p^2},c\in V$. A Jacobiator summand $[[e_i,e_j],e_l]$ with distinct
entries is therefore nonzero only if $[e_i,e_j]$ has a $u_t$-component and
$e_l\in\{x,y\}$. Checking the three cyclic summands for the triples
$\{x,y,u_i\}$, $\{x,u_i,u_j\}$, $\{y,u_i,u_j\}$, $\{x,y,x_p\}$,
$\{x,u_i,x_p\}$, $\{y,u_i,x_p\}$ gives in each case a sum of brackets of the
forms $[u_s,u_t]$, $[c,\,\cdot\,]$, $[y_{p^2},\,\cdot\,]$, $[y,u_s]$ with
$s\ge2$ all zero.

The verification above applies verbatim for $L$, with every occurrence of $c$
in the adjoint values deleted and with the Jacobi checks only losing
summands.

To conclude, we have $u_1=[x,y]$,
$u_{i+1}=[x,u_i]$, $x_p=x^{[p]}$, $y_p=y^{[p]}$, the $y$-chain by iterated
$p$-powers, and $c=x_p^{[p]}$, and likewise in $L$. Finally
$L'=\spann{v_1,\dots,v_{p-1},b_{p^2}}$ has dimension $p$, while in $H$ the
bracket $[y,u_1]=c$ adds the top:
$H'=\spann{u_1,\dots,u_{p-1},y_{p^2},c}$ has dimension $p+1$.
\end{proof}

\begin{corollary}\label{cor:nonisop}
$L\not\cong H$ as restricted Lie algebras, since
$\dim L'=p\neq p+1=\dim H'$.
\end{corollary}

We now produce the isomorphism. All computations take place in $u(H)$. By
Lemma~\ref{lem:pres} we may use freely $x^p=x_p$,
$x_p^{\,p}=c$, $y^p=y_p$, $y_{p^j}^{\,p}=y_{p^{j+1}}$, $y_{p^{k-1}}^{\,p}=c$,
$c^p=0$ and $u_i^{\,p}=0$. Iterating $y_{p^j}^{\,p}=y_{p^{j+1}}$ along
the chain gives $y_{p^2}^{\,p^j}=y_{p^{2+j}}$ for $0\le j\le k-2$, $y_{p^k}\coloneqq c$. This yields
\begin{equation}\label{eq:tower}
y_{p^2}^{\,q}=c,\qquad y_{p^2}^{\,pq}=c^{\,p}=0,
\qquad\text{hence}\qquad
y_{p^2}^{\,m}=0\ \text{ for every } m\ge pq .
\end{equation}

\begin{lemma}\label{lem:toolp}
In $u(H)$, with $\ t\coloneqq x_p y_{p^2}^{\,q-1}$ and
$\alpha\coloneqq 2^{-1}x_p^{\,2}\,y_{p^2}^{\,q-2}$:
\begin{enumerate}
\item $y_p,y_{p^2},\dots,y_{p^{k-1}},c$ are central, $x_p$ commutes with $x$,
      all $u_i$ and with all central elements, and $[x_p,y]=y_{p^2}$;
\item $[\alpha,y]=t$ and $[t,y]=c$;
\item $\alpha$ and $t$ commute with $x$, all $u_i$ and each other, and
      $\alpha^p=t^p=0$.
\end{enumerate}
\end{lemma}

\begin{proof}
(1) These are read off the bracket table of $H$. Since
$H$ generates $u(H)$, an element central in $H$ is central in $u(H)$.

(2) By direct computations using centrality,
\[
[\alpha,y]=2^{-1}y_{p^2}^{\,q-2}\,[x_p^{\,2},y]
=2^{-1}y_{p^2}^{\,q-2}\bigl(x_p[x_p,y]+[x_p,y]x_p\bigr)
=2^{-1}y_{p^2}^{\,q-2}\cdot2\,x_p\,y_{p^2}=x_p\,y_{p^2}^{\,q-1}=t,
\]
using that $[x_p,y]=y_{p^2}$ is central. Likewise $[t,y]=[x_p,y]\,y_{p^2}^{\,q-1}
=y_{p^2}\, y_{p^2}^{\,q-1}=y_{p^2}^{\,q}=c$.

(3) Commutation is clear, as $\alpha$ and $t$ are polynomials in
$x_p$ and central elements. For the $p$-th powers, $x_p$ commutes with the
central factors, so by \eqref{eq:square}
$t^p=x_p^{\,p}\,y_{p^2}^{\,p(q-1)}=c\,y_{p^2}^{\,pq-p}
=y_{p^2}^{\,q+pq-p}=0$ and
$\alpha^p=2^{-p}x_p^{\,2p}\,y_{p^2}^{\,p(q-2)}=2^{-p}c^2\,y_{p^2}^{\,pq-2p}
=2^{-p}\,y_{p^2}^{\,2q+pq-2p}=0$, since $q\ge p$ makes both exponents
greater or equal than $pq$.
\end{proof}

\begin{theorem}\label{thm:oddp}
For every odd prime $p$ and every $k\ge3$ the assignment
\[
\varphi(a)=X\coloneqq  x + \alpha = x+2^{-1}x_p^{\,2}\,y_{p^2}^{\,p^{k-2}-2},\qquad
\varphi(b)=Y\coloneqq y
\]
extends to an isomorphism of $\Fp$-algebras
$\varphi\colon u(L)\to u(H)$.
\end{theorem}

\begin{proof}
Write $X=x+\alpha$ with $\alpha$ as in Lemma~\ref{lem:toolp}. Define images
for the whole basis of $L$ as the corresponding words in $X$ and $Y$:
\begin{gather*}
f_a\coloneqq X,\qquad f_b\coloneqq Y,\qquad f_{v_1}\coloneqq [X,Y],\qquad f_{v_{i+1}}\coloneqq [X,f_{v_i}],\\
f_{a_p}\coloneqq X^p,\qquad f_{b_p}\coloneqq Y^p,\qquad
f_{b_{p^{j+1}}}\coloneqq f_{b_{p^j}}^{\,p},\qquad f_d\coloneqq f_{a_p}^{\,p}.
\end{gather*}
By Lemma~\ref{lem:pres}, applied to $L$, it suffices to check that these
elements satisfy all relations~\eqref{eq:pres} of~$u(L)$.

We begin by computing the explicit values of the images of the elements of $L$. 
\begin{align*}
f_{v_1}&=[x+\alpha,y]=u_1+[\alpha,y]=u_1+t, \\[2pt] 
f_{v_2}&=[X,u_1+t]=[x,u_1]=u_2,\\[2pt]
f_{v_{i+1}}&=[X,u_i]=[x,u_i]=u_{i+1}\ \ (2\le i\le p-2).
\end{align*}
Since $x$ and $\alpha$ commute and $\alpha^p=0$, we get
$X^p=x^p+\alpha^p=x_p$, so $f_{a_p}=x_p$ and
$f_d=x_p^{\,p}=c$. Similarly $f_{b_{p^j}}=y_{p^j}$ for all $j$, and the two
routes to the top agree:
$f_{a_p}^{\,p}=c=f_{b_{p^{k-1}}}^{\,p}$, matching
$a_p^{[p]}=d=b_{p^{k-1}}^{[p]}$ in $L$.

To check the brackets, we begin by noticing that the images of the chain elements $b_{p^j},d$
of $L$ are the central elements $y_{p^j},c$ of $u(H)$, so every bracket
relation involving them holds, both sides being zero. The remaining ones:
\begin{align*}
[f_a,f_{v_i}]&=f_{v_{i+1}}\ (i<p-1) \\[2pt]
[f_a,f_{v_{p-1}}]&=[x+\alpha,u_{p-1}]=y_{p^2}=f_{b_{p^2}} \\[2pt]
[f_{a_p},f_b]&=[x_p,y]=y_{p^2}=f_{b_{p^2}} \\[2pt]
[f_a,f_{a_p}]&=[x+\alpha,x_p]=0 \\[2pt]
[f_{v_1},f_{v_j}]&=[u_1+t,u_j]=0, [f_{v_1},f_{a_p}]=[u_1+t,x_p]=0 \\[2pt]
[f_{v_i},f_{v_j}]&=[f_{v_i},f_{a_p}]=0\ (i,j\ge2) \\[2pt]
[f_b,f_{v_1}]&=[y,\,u_1+t]=[y,u_1]+[y,t]=c+(-c)=0 \\[2pt]
[f_b,f_{v_i}]&=[y,u_i]=0\ (i\ge2)
\end{align*}
In each case the value coincides with the image of the corresponding bracket
of $L$. 

Finally, the power relations follow automatically: $f_a^{\,p}=X^p=x_p=f_{a_p}$ and
$f_b^{\,p}=y^p=y_p=f_{b_p}$, while the chains
$f_{b_{p^j}}^{\,p}=f_{b_{p^{j+1}}}$, $f_{a_p}^{\,p}=f_d$ and
$f_d^{\,p}=c^p=0$ match
$d^{[p]}=0$. For $i\ge2$, $f_{v_i}^{\,p}=u_i^{\,p}=0$. Lastly, $u_1$ and $t$
commute, so $f_{v_1}^{\,p}=(u_1+t)^p=u_1^{\,p}+t^p=0$, matching $v_1^{[p]}=0$ in $L$.

All relations \eqref{eq:pres} of $u(L)$ hold, so $\varphi$ is a well-defined
algebra homomorphism.

As before, since both algebras have dimension $p^{\,n}$, it suffices to show that $\varphi$ is surjective. 
Since $L$ is generated by $a,b$,
the image of $\varphi$ is the subalgebra $B$ generated by $X=x+\alpha$ and
$Y=y$, and it is enough to show $x,y\in B$. Trivially $y=Y\in B$, whence
$y_{p^2}=y^{p^2}=Y^{p^2}\in B$. Since $x$ and $\alpha$ commute and
$\alpha^p=0$, we have that 
$x_p=x^p=(x+\alpha)^p=X^p\in B$. Therefore
$\alpha=2^{-1}x_p^{\,2}\,y_{p^2}^{\,q-2}$, a polynomial in $x_p$ and
$y_{p^2}$, lies in $B$, and so does $x=X-\alpha$. Thus $B=u(H)$, and
$\varphi$ is an isomorphism.
\end{proof}

\section{Arbitrary fields}\label{sec:fields}

In order to generalise the examples to any field, we need only the restricted-Lie analogue of the classical
isomorphisms $U(\mathfrak g\otimes_kK)\cong U(\mathfrak g)\otimes_kK$
(see~\cite{SF}). It follows from the universal
property of $u$ together with the Poincar\'e--Birkhoff--Witt theorem.

\begin{proposition}\label{prop:base}
Let $M$ be a finite-dimensional restricted Lie algebra over $\Fp$ and $F$ a
field of characteristic $p$. Then $M\otimes_{\Fp}F$ carries a unique
restricted structure extending that of $M$, the natural map
$u(M)\otimes_{\Fp}F\to u(M\otimes_{\Fp}F)$ is an isomorphism of $F$-algebras,
and $(M\otimes_{\Fp}F)'=M'\otimes_{\Fp}F$.
\end{proposition}

We are now ready to prove the main theorem.
\begin{theorem}\label{thm:final}
Let $F$ be a field of characteristic $p>0$. For every integer $n\ge p+5$
there exist $p$-nilpotent restricted Lie algebras $L$ and $H$ over $F$ with
$\dim L=\dim H=n$, such that:
\[
u(L)\cong u(H)\ \text{ as $F$-algebras},\qquad
\dim L'=p\neq p+1=\dim H' .
\]
\end{theorem}

\begin{proof}
Given $n\ge p+5$, set $k\coloneqq n-4$ if $p=2$ and $k\coloneqq n-p-2$ if $p$ is odd, so that
$k\ge3$, and let $(L_0,H_0)$ be the pair constructed above over the prime field,
namely $(L_k,H_k)$ if $p=2$ and $(L_{p,k},H_{p,k})$ if $p$ is odd. Put
$L\coloneqq L_0\otimes_{\Fp}F$ and $H\coloneqq H_0\otimes_{\Fp}F$. By
Theorem~\ref{thm:char2} or~\ref{thm:oddp} and Proposition~\ref{prop:base},
\[
u(L)\cong u(L_0)\otimes_{\Fp}F\cong u(H_0)\otimes_{\Fp}F\cong u(H),
\]
while Lemma~\ref{lem:valid2} or~\ref{lem:validp},
together with Proposition~\ref{prop:base}, gives $\dim L'=p\neq p+1=\dim H'$.

These algebras are $p$-nilpotent. To see this, order the basis of $H_k$ as
$c,y_q,\dots,y_4,w,x_2,y_2,y,x$ and that of $H_{p,k}$ as
$c,y_{p^{k-1}},\dots,y_{p^2},x_p,y_p,u_{p-1},\dots,u_1,y,x$ (and analogously
for $L_k$ and $L_{p,k}$). Writing $z_1,\dots,z_n$ for this ordered basis and letting
$I_j\coloneqq \spann{z_1,\dots,z_j}$, the defining tables show that every bracket
$[\,\cdot\,,z_j]$ and every power $z_j^{[p]}$ lies in $I_{j-1}$. Hence for
$m=\sum_{i\le j}\lambda_iz_i\in I_j$, Jacobson's
formula~\eqref{eq:jacobsonformula} gives
$m^{[p]}=\sum_{i\le j}\lambda_i^{\,p}z_i^{[p]}+\Lambda\in I_{j-1}$, where
$\Lambda$ is a sum of iterated brackets of the summands $\lambda_iz_i\in I_j$.
Iterating, $m^{[p]^n}\in I_0=0$. The scalar extensions are $p$-nilpotent by
the same computation on the same ordered basis.
\end{proof}

\section{Concluding remarks}\label{sec:concl}

\subsection*{The examples among the determined invariants}

We first situate the counterexamples against the determined invariants.
If $u(L)\cong u(H)$, then by~\cite{UsefiPAMS,UsefiSurvey} the
dimension-subalgebra quotients $D_i(L)/D_{i+1}(L)$ and $D_i(H)/D_{i+1}(H)$
are isomorphic, so the associated graded restricted Lie algebras $\gr L$
and $\gr H$ are isomorphic, the nilpotency classes of $p$-nilpotent
algebras differ by at most one, and by~\cite{UsefiPacific} the quotients
$L/({L'}^{p}+\gamma_3(L))$ and $H/({H'}^{p}+\gamma_3(H))$ are isomorphic. All
of this necessarily holds for our pairs. What the present examples show is
that the determinacy stops there: the dimension of the derived subalgebra,
the most basic invariant not captured by the graded data, is already lost in
$u(L)$, in every positive characteristic.

\subsection*{The non-restricted enveloping algebra}

The phenomenon is genuinely restricted. Forgetting the $p$-map, $L$ and $H$
are ordinary Lie algebras, and one may ask whether $U(L)\cong U(H)$. In prime
characteristic this problem is itself known to have a negative answer
\cite{RileyUsefi,SchneiderUsefi}, but our pairs are not of that kind. This is
immediate from the standard fact that, for any Lie algebra~$\mathfrak g$, the
abelianisation of its enveloping algebra is
$U(\mathfrak g)^{\mathrm{ab}}\cong U(\mathfrak g/\mathfrak g')=S(\mathfrak g/\mathfrak g')$,
a polynomial algebra in $\dim(\mathfrak g/\mathfrak g')$ variables. As both
the abelianisation and its number of variables are isomorphism invariants,
$U(\mathfrak g)$ determines $\dim(\mathfrak g/\mathfrak g')$. The members of
each of our pairs have equal dimension but derived subalgebras of dimensions
$p$ and $p+1$, so $\dim(L/L')\ne\dim(H/H')$ and hence $U(L)\not\cong U(H)$.

In the restricted case the same computation reads instead
$u(L)^{\mathrm{ab}}\cong u(L/I)$, where $I$ is the $p$-ideal generated by $L'$:
the relations $x^p=x^{[p]}$ collapse the whole $p$-closure of $L'$ in the
commutative quotient. For our pairs $L'$ and $H'$ differ in dimension, yet the
$p$-ideals they generate have the same dimension. This is exactly what allows
$u(L)\cong u(H)$ while $U(L)\not\cong U(H)$.

\subsection*{How the examples were found}

The examples of this article were discovered by a computer search guided by
the determined invariants. 
The code used for this search was developed with the assistance of \texttt{Claude Code}.
Let us describe the process briefly. The mathematical arguments and final verification are independent of the
computer search.

The determined invariants dictate the shape of the search space. Since the
dimension-subalgebra quotients are determined by
$u(L)$, the two members of a counterexample pair must share
all their graded invariants, and the search must therefore take place among
\emph{filtered deformations}: we fixed, for each small dimension, a graded
restricted Lie algebra over $\Ft$ (an assignment of weights to a basis,
compatible with the towers of the $2$-map and with a generating
bracket), and parametrised its deformations, allowing each entry of the
bracket table and of the $2$-map to be perturbed by terms of strictly
larger weight, subject to the Jacobi identity and Jacobson's criterion.

Candidate pairs $(L,H)$ of deformations of one graded algebra were then
screened by invariants: a pair is of interest only if $L\not\cong H$. In
practice we started by using the cheap certificate $\dim L'\neq\dim H'$, while the determined invariants, such
as the dimension-subalgebra quotients, agree. For each surviving pair we
searched for an isomorphism $u(L)\to u(H)$ directly. In dimension 7 the
sampling produced the pair $(L_3,H_3)$ of Section~\ref{sec:char2}. 

After checking that the counterexample did not generalise easily to odd characteristic, we proceeded in a similar fashion,
setting $p=3$ and dimension 8, where we found the first pair that led us to the full generalisation.

\section*{Acknowledgements}
We would like to thank Diego García Lucas for very fruitful conversations about this problem.


\begin{thebibliography}{99}

\bibitem{Brauer}
R.~Brauer, \emph{Representations of finite groups}, in: Lectures on Modern
Mathematics, Vol.~I, Wiley, New York, 1963, pp.~133--175.

\bibitem{CPRW}
R.~Campos, D.~Petersen, D.~Robert-Nicoud and F.~Wierstra, \emph{Lie,
associative and commutative quasi-isomorphism}, Acta Math.\ \textbf{233}
(2024), no.~2, 195--238.

\bibitem{Deskins}
W.~E. Deskins, \emph{Finite Abelian groups with isomorphic group algebras},
Duke Math.\ J.\ \textbf{23} (1956), 35--40.

\bibitem{GLMdR}
D.~Garc\'ia-Lucas, L.~Margolis and \'A.~del~R\'io, \emph{Non-isomorphic
$2$-groups with isomorphic modular group algebras}, J.~Reine Angew.\ Math.\
\textbf{783} (2022), 269--274.

\bibitem{Hertweck}
M.~Hertweck, \emph{A counterexample to the isomorphism problem for integral
group rings}, Ann.\ of Math.\ (2) \textbf{154} (2001), no.~1, 115--138.

\bibitem{Jacobson}
N.~Jacobson, \emph{Lie Algebras}, Interscience Tracts in Pure and Applied
Mathematics \textbf{10}, Interscience Publishers, New York--London, 1962.

\bibitem{Jennings}
S.~A. Jennings, \emph{The structure of the group ring of a $p$-group over a
modular field}, Trans.\ Amer.\ Math.\ Soc.\ \textbf{50} (1941), 175--185.

\bibitem{Margolis}
L.~Margolis, \emph{The modular isomorphism problem: a survey}, Jahresber.\
Dtsch.\ Math.-Ver.\ \textbf{124} (2022), no.~3, 157--196.

\bibitem{PassiSehgal}
I.~B.~S. Passi and S.~K. Sehgal, \emph{Isomorphism of modular group
algebras}, Math.\ Z.\ \textbf{129} (1972), 65--73.

\bibitem{Quillen}
D.~G. Quillen, \emph{On the associated graded ring of a group ring},
J.~Algebra \textbf{10} (1968), 411--418.

\bibitem{RileyUsefi}
D.~M. Riley and H.~Usefi, \emph{The isomorphism problem for universal
enveloping algebras of Lie algebras}, Algebr.\ Represent.\ Theory
\textbf{10} (2007), no.~6, 517--532.

\bibitem{Sandling}
R.~Sandling, \emph{The isomorphism problem for group rings: a survey}, in:
Orders and their Applications (Oberwolfach, 1984), Lecture Notes in Math.\
\textbf{1142}, Springer, Berlin, 1985, pp.~256--288.

\bibitem{SchneiderUsefi}
C.~Schneider and H.~Usefi, \emph{The isomorphism problem for universal
enveloping algebras of nilpotent Lie algebras}, J.~Algebra \textbf{337}
(2011), 126--140.

\bibitem{SF}
H.~Strade and R.~Farnsteiner, \emph{Modular Lie Algebras and their
Representations}, Monographs and Textbooks in Pure and Applied Mathematics
\textbf{116}, Marcel Dekker, New York, 1988.

\bibitem{UsefiPAMS}
H.~Usefi, \emph{The restricted isomorphism problem for metacyclic restricted
Lie algebras}, Proc.\ Amer.\ Math.\ Soc.\ \textbf{136} (2008), no.~12,
4125--4133.

\bibitem{UsefiPacific}
H.~Usefi, \emph{Isomorphism invariants of restricted enveloping algebras},
Pacific J.~Math.\ \textbf{246} (2010), no.~2, 487--494.

\bibitem{UsefiSurvey}
H.~Usefi, \emph{Isomorphism invariants of enveloping algebras}, in:
Noncommutative Rings and their Applications, Contemp.\ Math.\ \textbf{634},
Amer.\ Math.\ Soc., Providence, RI, 2015, pp.~253--265.

\end{thebibliography}
\end{document}